\documentclass[11pt,reqno]{amsart}

\newtheorem{proposition}{Proposition}[section]
\newtheorem{lemma}[proposition]{Lemma}
\newtheorem{corollary}[proposition]{Corollary}
\newtheorem{theorem}[proposition]{Theorem}

\theoremstyle{definition}

\newtheorem{remark}[proposition]{Remark}

\newcommand{\thlabel}[1]{\label{th:#1}}
\newcommand{\thref}[1]{Theorem~\ref{th:#1}}
\newcommand{\selabel}[1]{\label{se:#1}}
\newcommand{\seref}[1]{Section~\ref{se:#1}}
\newcommand{\lelabel}[1]{\label{le:#1}}
\newcommand{\leref}[1]{Lemma~\ref{le:#1}}

\newcommand{\colabel}[1]{\label{co:#1}}
\newcommand{\coref}[1]{Corollary~\ref{co:#1}}

\newcommand{\eqlabel}[1]{\label{eq:#1}}
\newcommand{\equref}[1]{(\ref{eq:#1})}

\def\lan{\langle}
\def\ran{\rangle}
\def\ot{\otimes}

\def\NN{{\mathbb N}}

\def\ZZ{{\mathbb Z}}

\newcommand{\Cc}{\mathcal{C}}

\def\*C{{}^*\hspace*{-1pt}{\Cc}}
\def\text#1{{\rm {\rm #1}}}

\input xy
\xyoption {all} \CompileMatrices

\usepackage{amssymb}
\usepackage{color,amssymb,graphicx,amscd}
\usepackage{nameref}
\usepackage[colorlinks,urlcolor=blue,linkcolor=blue,citecolor=blue]{hyperref}

\begin{document}

\title[Classifying bicrossed products of two Taft algebras]
{Classifying bicrossed products of two Taft algebras}

\author{A. L. Agore}
\address{Faculty of Engineering, Vrije Universiteit Brussel, Pleinlaan 2, B-1050 Brussels, Belgium \textbf{and} ''Simion Stoilow'' Institute of Mathematics of the Romanian Academy, P.O. Box 1-764, 014700 Bucharest, Romania}
\email{ana.agore@vub.ac.be and ana.agore@gmail.com}

\subjclass[2010]{16T10, 16T05, 16S40}

\thanks{The author is postdoctoral fellow of the Fund for Scientific Research-Flanders (Belgium) (FWO-Vlaanderen).}

\subjclass[2010]{16T05, 16S40} \keywords{Bicrossed product, The
factorization problem, Classification of Hopf algebras, Taft
algebra, Automorphism group, Drinfel'd double.}


\begin{abstract}
We classify all Hopf algebras which factor through two Taft
algebras $\mathbb{T}_{n^{2}}(\overline{q})$ and respectively
$T_{m^{2}}(q)$. To start with, all possible matched pairs between
the two Taft algebras are described: if $\overline{q} \neq
q^{n-1}$ then the matched pairs are in bijection with the group of
$d$-th roots of unity in $k$, where $d = (m,\,n)$ while if
$\overline{q} = q^{n-1}$ then besides the matched pairs above we
obtain an additional family of matched pairs indexed by $k^{*}$.
The corresponding bicrossed products (double cross product in Majid's terminology) are explicitly described by
generators and relations and classified. As a consequence of our
approach, we are able to compute the number of isomorphism types
of these bicrossed products as well as to describe their
automorphism groups.
\end{abstract}

\maketitle

\section*{Introduction}

The factorization problem originates in group theory and it was
first considered by Maillet (\cite{Mai}). Since then, it was also
introduced and intensively studied in the context of other
mathematical objects such as: (co)algebras (\cite{Tom, CIMZ}), Lie
algebras (\cite{Mic}), Lie groups (\cite{majid3}), groupoids
(\cite{andr}), Hopf algebras (\cite{majid2}), fusion categories (\cite{Gk}) and so on.  For a
detailed historical update on the problem we refer to \cite{abm,
am2015} and the references therein. The factorization problem in
its original group setting asks for the description and
classification of all groups $X$ which factor through two given
groups $G$ and $H$, i.e $X = GH$ and $G \cap H = \{1\}$. However,
although the statement of the problem is very simple and natural,
no major progress has been made so far as we still lack exhaustive
methods to tackle it. For instance, even the description and
classification of groups which factor through two finite cyclic
groups is still an open problem although there are several papers
dealing with it, such as the four papers by J. Douglas
\cite{Douglas} and the more recent one \cite{acim} which solves
completely the problem in the special case where one of the groups
is of prime order.

One turning point in studying the factorization problem for groups
was the bicrossed product construction introduced in a paper by
Zappa (\cite{zappa}); later on, the same construction appears in a
paper of Takeuchi \cite{Takeuchi} where the terminology bicrossed
product originates. The main ingredients in constructing bicrossed
products are the so-called matched pairs of groups. The
corresponding notions in the context of Hopf algebras (see
\seref{1.2} for the precise definitions) were introduced by Majid
(\cite{majid}) and allowed for a more computational approach of
the problem. This point of view was also considered recently in
\cite{abm} where a strategy for classifying bicrossed products of
Hopf algebras was proposed. The line of inquiry presented in
\cite{abm} was followed in \cite{Gabi} where the Hopf algebras
which factor through two Sweedler's Hopf algebras are described
and classified as well as in \cite{marc1} where the automorphism
group of the Drinfel'd double of a purely non-abelian finite group
is completely described. Similar ideas are also employed in
\cite{marc2} in order  to determine all quasitriangular structures
and ribbon elements on the Drinfel'd double of a finite group over
an arbitrary field.

In this paper, using the method presented in \cite{abm}, we
investigate Hopf algebras which factor through two Taft
algebras $T_{m^{2}}(q)$ and respectively
$\mathbb{T}_{n^{2}}(\overline{q})$. More precisely, we will
describe and classify all bicrossed products $T_{m^{2}}(q) \bowtie
\mathbb{T}_{n^{2}}(\overline{q})$. The number of isomorphism types
of these bicrossed products is also computed and, as expected, it
depends heavily on the arithmetics of the base field $k$. In
particular, for $m = n$ and $\overline{q} = q^{n-1}$, we find
the celebrated Drinfel'd double $D\bigl(T_{n^{2}}(q) \bigl)$ which
however is just one of the bicrossed products between two Taft
algebras having the same dimension. As a consequence of our
strategy the automorphism groups of all bicrossed products
$T_{m^{2}}(q) \bowtie \mathbb{T}_{n^{2}}(\overline{q})$ are
explicitly described. We mention as well that the problem of
describing the automorphism group of a given Hopf algebra is a
notoriously difficult problem coming from invariant theory.

The paper is organized as follows. In \nameref{prel} we review the
construction of the bicrossed product associated to a matched pair
of Hopf algebras $(A, H, \triangleleft, \triangleright)$. Then, we
state Majid's result (\thref{carMaj}) which turns the
factorization problem for Hopf algebras into a computational one:
given $A$ and $H$ two Hopf algebras, describe the set of all
matched pairs $(A, H, \triangleright, \triangleleft)$ and classify
up to an isomorphism all bicrossed products $A \bowtie \, H$.
\seref{main} gathers our main results. We start by describing all
matched pairs $(\mathbb{T}_{n^{2}}(\overline{q}), \,
T_{m^{2}}(q),\, \triangleright, \, \triangleleft)$ in
\thref{main}. It turns out that if $\overline{q} \neq q^{n-1}$
then the matched pairs are in bijection with the group of $d$-th
roots of unity in $k$, where $d = (m,\,n)$; if $\overline{q} =
q^{n-1}$, then besides the matched pairs above we obtain an
additional family of matched pairs indexed by $k^{*}$. The
bicrossed products corresponding to the two families of matched
pairs from \thref{main} which we denote by $\mathrm{T}_{n,\,m}^{\,
\sigma}(\overline{q},\,q)$ and respectively
$\mathrm{Q}_{n}^{\alpha}(q)$, for some $\sigma \in U_{d}(k)$ and
$\alpha \in k^{*}$, are described by generators and relations in
\coref{descriere}. Our main classification result is
\thref{clasificare} which gives necessary and sufficient
conditions for any two bicrossed products between
$\mathbb{T}_{n^{2}}(\overline{q})$ and $T_{m^{2}}(q)$ to be
isomorphic. Several cases need to be considered depending on
whether $n \neq m$ or $m=n$ and respectively $\overline{q} \neq q$
or $\overline{q} = q$. In particular, it is proved that the
Drinfel'd double $D\bigl(T_{n^{2}}(q)\bigl)$ is isomorphic to
$\mathrm{Q}_{n}^{1}(q)$.

\thref{clasificare} has two very important consequences. The first
one is \coref{nr} which records the number of isomorphism types of
the Hopf algebras described in \coref{descriere} while the second
one, \thref{auto}, describes explicitly their automorphism groups.
For instance the automorphism group of the Drinfel'd double
$D\bigl(T_{n^{2}}(q) \bigl)$ is proved to be isomorphic to $k^{*}$
for $n \geqslant 3$ and to a semidirect product $k^{*} \rtimes
\ZZ_{2}$ for $n=2$ (see also \cite{Gabi}).

\section{Preliminaries}\label{prel}
Throughout this paper, $k$ will be an arbitrary field of
characteristic zero. Unless otherwise specified, all algebras,
coalgebras, bialgebras, Hopf algebras, tensor products and
homomorphisms are over $k$. For a coalgebra $C$, we use Sweedler's
$\Sigma$-notation: $\Delta(c) = c_{(1)}\ot c_{(2)}$,
$(I\ot\Delta)\Delta(c) = c_{(1)}\ot c_{(2)}\ot c_{(3)}$, etc
(summation understood). If $C$ is a coalgebra, the opposite
coalgebra, $C^{cop}$, has the same underlaying vector space $C$
but with the comultiplication given by $\Delta^{cop} = \tau \circ
\Delta$, where $\tau$ is the flip map $\tau (a \otimes b) = b
\otimes a$. Given any vector space $V$ we denote by $V^{\ast}$ its
dual space; if $v \in V$, then the corresponding element of
$V^{\ast}$ will be denoted simply by $v^{\ast}$. The notation
$\delta_{i,\,j}$ stands for the Kronecker delta.

For an integer $n$ we denote by $U_n (k) = \{ \omega \in k \, | \,
\omega^n = 1 \}$ the cyclic group of $n$-th roots of unity in $k$.
If $| U_n (k) | = n$, then any generator of $U_n (k)$ is called a
primitive $n$-th root of unity; we will denote by $P_{n}(k)$ the
set of all primitive $n$-th roots of unity. It is straightforward
to see that if $m \neq n$ then $P_{m}(k) \cap P_{n}(k) =
\varnothing$. For any two positive integers $m$ and $n$, we will
denote their greatest common divisor by $(m,\,n)$.

Given a fixed positive integer $m \in \NN^{*}$ and $q$  a primitive
$m$-th root of unity in $k$ we will denote by $T_{m^{2}}(q)$ the Taft
Hopf algebra of order $m^{2}$. Recall that $T_{m^{2}}(q)$ is
generated as an algebra by two elements $h$ and $x$ subject to the
relations $h^{m} = 1$, $x^{m} = 0$ and $xh = qhx$. The coalgebra
structure is as follows: $h$ is a group like element while $x$ is
$(h,1)$-primitive. That is, the comultiplication, counit and
respectively antipode are given by:
\begin{eqnarray*}
\Delta(h) = h \ot h, \,\, \Delta(x) = x \ot h + 1 \ot x, \,\, \epsilon(h) = 1, \\
\epsilon(x) = 0,\,\ S(h) = h^{-1},\,\, S(x) = -xh^{-1}.
\end{eqnarray*}
Sweedler's Hopf algebra is obtained by considering $m=2$ and $q =
-1$. It is well known that $\{h^{i}x^{j}\}_{0 \leq i,\,j \leq
m-1}$ is a $k$-linear basis of the Taft algebra, the set of
group-like elements is $\mathcal{G}\bigl(T_{m^{2}}(q)\bigl) = \{1,
h, h^{2}, \cdots, h^{m-1}\}$ and the primitive elements
$\mathcal{P}_{h^{j},\, 1}\bigl(T_{m^{2}}(q)\bigl)$ are given as
follows for any $j = \{0,\,1, \cdots,\, m-1\}$ (see for instance
\cite{Sch}):
$$\mathcal{P}_{h^{j},\, 1} \bigl(T_{m^{2}}(q)\bigl) = \left \{\begin{array}{rcl}
\alpha(h^{j} - 1), \, & \mbox { if }& j \neq 1\\
\beta(h-1)+ \gamma x, \, & \mbox { if }& j=1
\end{array} \right.,\,\, {\rm for}\,\, {\rm some}\,\, \alpha, \, \beta,\, \gamma \in k.
$$
Also, it is worth pointing out that two Taft algebras $T_{m^{2}}(q)$ and respectively $T_{m^{2}}(q')$ are  isomorphic as Hopf
algebras  if and only if $q =
q'$. Furthermore, we also have a Hopf algebra
isomorphism $\psi : T_{m^{2}}(q) \to T_{m^{2}}(q)^{\ast}$ defined
as follows for any $i$, $j \in  \{0,\,1, \cdots,\, m-1\}$:

$$\psi (h) = h^{\ast}, \,\,\, \psi(x) = x^{*}$$
\begin{equation}\eqlabel{dual}
h^{\ast}(h^{i}x^{j}) = q^{i} \delta_{j,\,0},\,\,
x^{\ast}(h^{i}x^{j}) = \delta_{j,\,1}.
\end{equation}
Another trivial remark which will be useful in the sequel is that
the antipode of the Taft algebra provides an isomorphism of Hopf
algebras between $T_{m^{2}}(q)^{cop}$ and $T_{m^{2}}(q^{m-1})$.

Since we will be working with two Taft algebras (sometimes having
the same dimension) in order to avoid any confusion the second
Taft algebra will be denoted by $\mathbb{T}_{n^{2}}(\overline{q})$
where $\overline{q} \in P_n (k)$, and by $H$ and $X$ its
generators. Therefore, $H^{n} = 1$, $X^{n} = 0$, $XH =
\overline{q} HX$, $H$ is a group-like element while $X$ is an
$(H,\, 1)$-primitive element.

\subsection*{Bicrossed product of Hopf algebras}\selabel{1.2}

We recall briefly the main characters of this paper, namely the
bicrossed products of Hopf algebras. We will adopt the terminology bicrossed product
from \cite[Theorem IX 2.3]{Kassel} as it is also employed for the
analogous construction performed at the level of groups
\cite{Takeuchi}; other term referring to this construction used in the
literature is double cross product \cite[Proposition 3.12]{majid}.
A \textit{matched pair} of Hopf algebras is a
system $(A, H, \triangleleft, \triangleright)$, where $A$ and $H$
are Hopf algebras, $\triangleleft : H \otimes A \rightarrow H$,
$\triangleright: H \otimes A \rightarrow A$ are coalgebra maps
such that $(A, \triangleright)$ is a left $H$-module coalgebra,
$(H, \triangleleft)$ is a right $A$-module coalgebra and the
following compatibilities hold for any $a$, $b\in A$, $g$, $h\in
H$:
\begin{eqnarray}
h \triangleright1_{A} &{=}& \varepsilon_{H}(h)1_{A}, \quad 1_{H}
\triangleleft a =
\varepsilon_{A}(a)1_{H} \eqlabel{mp1} \\
g \triangleright(ab) &{=}& (g_{(1)} \triangleright a_{(1)}) \bigl
( (g_{(2)}\triangleleft a_{(2)})\triangleright b \bigl)
\eqlabel{mp2} \\
(g h) \triangleleft a &{=}& \bigl( g \triangleleft (h_{(1)}
\triangleright a_{(1)}) \bigl) (h_{(2)} \triangleleft a_{(2)})
\eqlabel{mp3} \\
g_{(1)} \triangleleft a_{(1)} \otimes g_{(2)} \triangleright
a_{(2)} &{=}& g_{(2)} \triangleleft a_{(2)} \otimes g_{(1)}
\triangleright a_{(1)} \eqlabel{mp4}
\end{eqnarray}
Recall that $H$ is called a right $A$-module coalgebra if $H$ is a
coalgebra in the monoidal category ${\mathcal M}_A $ of right
$A$-modules, i.e. there exists $\triangleleft : H \otimes A
\rightarrow H$ a morphism of coalgebras such that $(H,
\triangleleft) $ is a right $A$-module. For further computations,
the fact that two $k$-linear maps $\triangleleft : H \otimes A
\rightarrow H$, $\triangleright: H \otimes A \rightarrow A$ are
coalgebra maps can be written explicitly as follows:
\begin{eqnarray}
\Delta_{H}(h \triangleleft a) &{=}& h_{(1)} \triangleleft a_{(1)}
\otimes h_{(2)} \triangleleft a_{(2)}, \,\,\,\,\,
\varepsilon_{A}(h \triangleleft a) =
\varepsilon_{H}(h)\varepsilon_{A}(a)
\eqlabel{6}\\
\Delta_{A}(h \triangleright a) &{=}& h_{(1)} \triangleright
a_{(1)} \otimes h_{(2)} \triangleright a_{(2)}, \,\,\,\,\,
\varepsilon_{A}(h \triangleright a) =
\varepsilon_{H}(h)\varepsilon_{A}(a) \eqlabel{8}
\end{eqnarray}
for all $h \in H$, $a \in A$. The actions $\triangleleft : H
\otimes A \rightarrow H$, $\triangleright: H \otimes A \rightarrow
A$ are called \emph{trivial} if $h \triangleleft a = \varepsilon_A
(a) h$ and respectively $h \triangleright a = \varepsilon_H(h) a$,
for all $a\in A$ and $h\in H$.

If $(A, H, \triangleleft, \triangleright)$ is a matched pair of
Hopf algebras, the \textit{bicrossed product} $A \bowtie H$ of $A$
with $H$ is the $k$-module $A\ot H$ with the multiplication given
by
\begin{equation}\eqlabel{0010}
(a \bowtie h) \cdot (c \bowtie g):= a (h_{(1)}\triangleright
c_{(1)}) \bowtie (h_{(2)} \triangleleft c_{(2)}) g
\end{equation}
for all $a$, $c\in A$, $h$, $g\in H$, where we denoted $a\ot h$ by
$a\bowtie h$. $A \bowtie H$ is a Hopf algebra with the coalgebra
structure given by the tensor product of coalgebras and the
antipode:
\begin{equation}\eqlabel{antipbic}
S_{A \bowtie H} ( a \bowtie h ) = S_H (h_{(2)}) \triangleright S_A
(a_{(2)}) \, \bowtie \, S_H (h_{(1)}) \triangleleft S_A (a_{(1)})
\end{equation}
for all $a\in A$ and $h\in H$.

The generic example of a bicrossed product is the Drinfel'd double
$D(H)$. For any finite dimensional Hopf algebra $H$ we have a
matched pair of Hopf algebras $( (H^*)^{\rm cop}, H,
\triangleleft, \triangleright)$, where the actions $\triangleleft$
and $\triangleright$ are given as follows:
\begin{equation} \eqlabel{dubluact}
h \triangleleft h^* := \lan h^*, \, S_H^{-1} (h_{(3)}) h_{(1)}
\ran h_{(2)}, \quad h \triangleright h^* := \lan h^*, \, S_H^{-1}
(h_{(2)}) \, ? \, h_{(1)} \ran
\end{equation}
for all $h\in H$ and $h^* \in H^*$, where the question mark should be seen as a mute variable (\cite[Theorem
IX.3.5]{Kassel}). The Drinfel'd double of $H$ is the bicrossed
product associated to this matched pair, i.e. $D(H) = (H^*)^{\rm
cop} \bowtie H$. A generalization of this construction for
infinite dimensional Hopf algebras was introduced by Majid under
the name of \textit{generalized quantum double} (see
(\cite[Example 7.2.6]{majid})).

A Hopf algebra $E$ \emph{factors} through two Hopf algebras $A$
and $H$ if there exists injective Hopf algebra maps $i : A \to E $
and $j : H\to E$ such that the map
$$
A \ot H \to E, \quad a \ot h \mapsto i(a) j(h)
$$
is bijective. The next crucial result due to Majid \cite[Theorem
7.2.3]{majid2} characterizes Hopf algebras which factor through
two Hopf subalgebras in terms of matched pairs and the
corresponding bicrossed products.

\begin{theorem}\thlabel{carMaj}
Let $A$, $H$ be two Hopf algebras. A Hopf algebra $E$ factors
through $A$ and $H$ if and only if there exists a matched pair of
Hopf algebras $(A, H, \triangleleft, \triangleright)$ such that $E
\cong A \bowtie H$, an isomorphism of Hopf algebras.
\end{theorem}

Our approach for classifying bicrossed products of Hopf algebras
relies on the strategy previously developed in \cite{abm}. For the
reader's convenience we state here the main classification result
from \cite{abm} which will be used throughout:

\begin{theorem}\thlabel{toatemorf}
Let $(A, H, \triangleright, \triangleleft)$ and $(A', H',
\triangleright', \triangleleft')$ be two matched pairs of Hopf
algebras. Then there exists a bijective correspondence between the
set of all morphisms of Hopf algebras $\psi : A \bowtie H \to A'
\bowtie ' H' $ and the set of all quadruples $(u, p, r, v)$, where
$u: A \to A'$, $p: A \to H'$, $r: H \rightarrow A'$, $v: H
\rightarrow H'$ are unitary coalgebra maps satisfying the
following compatibility conditions:
\begin{eqnarray}
u(a_{(1)}) \ot p(a_{(2)}) &{=}& u(a_{(2)}) \ot p(a_{(1)})\eqlabel{C1}\\
r(g_{(1)}) \ot v(g_{(2)}) &{=}& r(g_{(2)}) \ot v(g_{(1)})\eqlabel{C2}\\
u(ab) &{=}& u(a_{(1)}) \, \bigl( p (a_{(2)}) \triangleright' u(b) \bigl)\eqlabel{C3}\\
p(ab) &{=}& \bigl( p (a) \triangleleft' u(b_{(1)}) \bigl) \, p (b_{(2)})\eqlabel{C4}\\
r(tg) &{=}& r(t_{(1)}) \, \bigl(v(t_{(2)}) \triangleright'
r(g)\bigl)\eqlabel{C5}\\
v(tg) &{=}& \bigl(v(t) \triangleleft' r(g_{(1)})\bigl) \,
v(g_{(2)})\eqlabel{C6}\\
r(g_{(1)}) \bigl(v(g_{(2)}) \triangleright' u(b) \bigl) &{=}& u
(g_{(1)} \triangleright b_{(1)}) \, \Bigl( p (g_{(2)}
\triangleright b_{(2)}) \triangleright' r(g_{(3)} \triangleleft
b_{(3)})
\Bigl) \eqlabel{C7}\\
\bigl (v(g) \triangleleft' u(b_{(1)}) \bigl) \, p (b_{(2)}) &{=}&
\Bigl( p (g_{(1)} \triangleright b_{(1)}) \triangleleft '
r(g_{(2)} \triangleleft b_{(2)}) \Bigl) \, v (g_{(3)}
\triangleleft b_{(3)}) \eqlabel{C8}
\end{eqnarray}
for all $a$, $b \in A$, $g$, $t \in H$.

Under the above correspondence the morphism of Hopf algebras
$\psi: A \bowtie H \to A' \bowtie ' H' $ corresponding to $(u, p,
r, v)$ is given by:
\begin{equation}\eqlabel{morfbicros}
\psi(a \bowtie g) = u(a_{(1)}) \, \bigl( p(a_{(2)})
\triangleright' r(g_{(1)}) \bigl) \,\, \bowtie' \, \bigl(
p(a_{(3)}) \triangleleft' r(g_{(2)}) \bigl) \, v(g_{(3)})
\end{equation}
for all $a \in A$ and $h\in H$.
\end{theorem}

The following straightforward result will prove to be very effective in computing matched pairs of Hopf algebras.

\begin{lemma}\lelabel{primitive}
Let $(A, H, \triangleleft, \triangleright)$ be a matched pair of
Hopf algebras, $a$, $b \in G(A)$ and $g$, $h\in G(H)$. Then:

$(1)$ $g \triangleright a \in G(A)$ and $g\triangleleft a \in
G(H)$;

$(2)$ If $x \in P_{a, \, 1}(A)$, then $g \triangleleft x \in
P_{g\triangleleft a, \, g} (H)$ and $g
\triangleright x \in P_{g\triangleright a, \, 1}
(A)$;

$(3)$ If $y \in P_{g, \, 1}(H)$, then $y \triangleleft a \in
P_{g\triangleleft a, \, 1} (H)$ and $y
\triangleright a \in P_{g\triangleright a, \, a}
(A)$.
\end{lemma}

\section{The bicrossed products of two Taft algebras}\selabel{main}

This section contains our main results. We start by describing all
possible matched pairs between two Taft algebras.

\begin{theorem}\thlabel{main}
Let $m$, $n \in \NN^{*} \setminus \{1\}$, $d = (m,\, n)$ and $q \in P_{m}(k)$, $\overline{q} \in P_{n}(k)$.

\begin{enumerate}

\item[1)] If $\overline{q} \neq q^{n-1}$ then there is a bijective correspondence between the matched pairs $(\mathbb{T}_{n^{2}}(\overline{q}), \, T_{m^{2}}(q),\, \triangleright, \, \triangleleft)$  and $U_{d}(k)$ such that the matched pair corresponding to the $d$-th root of unity $\sigma$ is given as follows:

\begin{center}
\begin{tabular} {l | c  c  }
$\triangleleft$ & $H$ & $X$ \\
\hline $h$ & $h$ & $0$ \\
$x$ & $\sigma x$ & $0$ \\
\end{tabular} \qquad
\begin{tabular} {l | c  c  }
$\triangleright$ & $H$ & $X$ \\
\hline $h$ & $H$ & $\sigma X$ \\
$x$ & $0$ & $0$ \\
\end{tabular}
\end{center}

\item[2)] If $\overline{q} = q^{n-1}$ then in addition to the matched pair depicted above we also have the following matched pair for any $\alpha \in k^{*}$:

\begin{center}
\begin{tabular} {l | c  c  }
$\triangleleft$ & $H$ & $X$ \\
\hline $h$ & $h$ & $0$ \\
$x$ & $q x$ & $\alpha (1-h)$ \\
\end{tabular} \qquad
\begin{tabular} {l | c  c  }
$\triangleright$ & $H$ & $X$ \\
\hline $h$ & $H$ & $q X$ \\
$x$ & $0$ & $\alpha(1-H)$ \\
\end{tabular}
\end{center}
\end{enumerate}
\end{theorem}
\begin{proof}
Let $(\mathbb{T}_{n^{2}}(\overline{q}), \, T_{m^{2}}(q),\, \triangleright, \, \triangleleft)$ be a matched pair of Hopf algebras, where $ \triangleright : T_{m^{2}}(q) \ot \mathbb{T}_{n^{2}}(\overline{q}) \to \mathbb{T}_{n^{2}}(\overline{q})$ and $\triangleleft : T_{m^{2}}(q) \ot \mathbb{T}_{n^{2}}(\overline{q}) \to T_{m^{2}}(q)$. We will start by proving that $h \triangleright H = H$. First, notice that by \leref{primitive} we have $h \triangleright H \in \mathcal{G} \bigl(\mathbb{T}_{n^{2}}(\overline{q})\bigl) = \{1, H, \cdots, H^{n-1}\}$. It is easy to see that $h \triangleright H \neq 1$; indeed, if $h \triangleright H = 1$ then since $\mathbb{T}_{n^{2}}(\overline{q})$ is a left $T_{m^{2}}(q)$-module via $\triangleright$ we would obtain $H = h^{m} \triangleright H = 1$ which is a contradiction. Therefore, $h \triangleright H = H^{t}$ with $t \in \{1, 2, \cdots, n-1\}$. Suppose $h \triangleright H = H^{t}$ with $t \neq 1$. Then using again \leref{primitive} we obtain $h \triangleright X \in \mathcal{P}_{H^{t},\, 1}(\mathbb{T}_{n^{2}}(\overline{q}))$ so $h \triangleright X = \alpha(1 - H^{t})$. By induction we obtain $X = h^{m} \triangleright X = \alpha(1 - H^{t_{m}})$ for some $t_{m} \in \{0, 1, \cdots, n-1\}$ and we have reached a contradiction. Thus we need to have $h \triangleright H = H$ and $h \triangleright X \in \mathcal{P}_{H,\, 1}(\mathbb{T}_{n^{2}}(\overline{q}))$ hence $h \triangleright X = a(1 - H) + b X$ for some $a$, $b \in k$. Furthermore, this leads us to $x \triangleright H \in \mathcal{P}_{h \triangleright H,\, H}(\mathbb{T}_{n^{2}}(\overline{q})) = \mathcal{P}_{H,\, H}(\mathbb{T}_{n^{2}}(\overline{q}))$, i.e. $x \triangleright H = 0$.  Now since $ \triangleright : T_{m^{2}}(q) \ot \mathbb{T}_{n^{2}}(\overline{q}) \to \mathbb{T}_{n^{2}}(\overline{q})$  is in particular a coalgebra map we have:
\begin{eqnarray*}
\Delta (x \triangleright X) &\stackrel{\equref{8}}{=}& x_{(1)} \triangleright X_{(1)} \ot x_{(2)} \triangleright X_{(2)}\\
&{=}& x \triangleright X \ot h \triangleright H + x \triangleright 1 \ot h \triangleright X + 1 \triangleright X \ot x \triangleright H + 1 \triangleright 1 \ot x \triangleright X\\
&{=}& x \triangleright X \ot H + 1 \ot x \triangleright X
\end{eqnarray*}
thus $x \triangleright X \in \mathcal{P}_{H,\, 1}(\mathbb{T}_{n^{2}}(\overline{q}))$, i.e. $x \triangleright X = \alpha (1 - H) + \beta X$ for some $\alpha$, $\beta \in k$.

Now we check under what conditions the two actions are compatible with the relations satisfied by the generators of the two Taft algebras. To start with, it is obvious that $x^{m} \triangleright H = 0$. By induction one can easily prove that $h^{i}  \triangleright X = a(1 + b + \cdots + b^{i-1}) (1 - H) + b^{i} X$ for all $i \in \{1, 2, \cdots, m\}$. As $h^{m} = 1$ we get:
\begin{eqnarray}
b^{m} = 1, \qquad a (1 + b + \cdots + b^{m-1}) = 0.
\end{eqnarray}
Furthermore, we have:
\begin{eqnarray*}
xh \triangleright X &=& b \alpha - b \alpha H + b \beta X\\
qhx \triangleright X &=& q(\alpha + a \beta) - q(\alpha + a \beta) H + q b \beta X
\end{eqnarray*}
which yields
\begin{eqnarray*}
q(\alpha + a \beta) = b \alpha, \qquad q b \beta = b \beta.
\end{eqnarray*}
As $q \neq 1$ we obtain $b \beta = 0$ and since $b^{m} = 1$ it follows that $\beta = 0$. Finally, one can now easily check that $x^{m} \triangleright X = 0$. To summarize, so far we obtained:
\begin{eqnarray*}
h \triangleright H = H, \,\, x \triangleright H = 0, \,\, x \triangleright X = \alpha (1-H), \,\, h \triangleright  X = a(1 - H) + bX
\end{eqnarray*}
where $b^{m} =1$, $\alpha (q - b) = 0$ and $a (1 + b + \cdots + b^{m-1}) = 0$.

In the same manner, it can be proved that $(T_{m^{2}}(q) ,\, \triangleleft)$ is a right $\mathbb{T}_{n^{2}}(\overline{q})$ - module coalgebra if and only if:
\begin{eqnarray*}
h \triangleleft H = h, \,\, h \triangleleft X = 0, \,\, x \triangleleft X = \mu (1-h), \,\, x \triangleleft  H = \gamma (1 - h) + \sigma x
\end{eqnarray*}
where $\sigma^{n} =1$, $\mu (1 - \overline{q} \sigma) = 0$, $\gamma (1 + \sigma + \cdots + \sigma^{n-1}) = 0$.

Next we impose \equref{mp2} and \equref{mp3} to be fulfilled by the pairs $(\triangleright, \, \triangleleft)$ depicted above. To start with, we can easily prove by induction that $h^{m}  \triangleleft H = h^{m}$ and $h^{m}  \triangleleft X = 0$ which are trivially fulfilled since $h^{m} = 1$. We also have:
\begin{eqnarray*}
xh \triangleleft H &=& \gamma h - \gamma h^{2} + \sigma xh \\
qhx \triangleleft H &=& q\gamma h -  q\gamma h^{2} + q \sigma hx
\end{eqnarray*}
which by \equref{mp3} implies $\gamma = q \gamma$ and since $q \neq 1$ we get $\gamma = 0$. Furthermore, the following identities:
\begin{eqnarray*}
xh \triangleleft X &=& b \mu h - b \mu h^{2} + (a - a \sigma) xh \\
qhx \triangleleft X &=& q \mu h -  q \mu h^{2}
\end{eqnarray*}
yield $a (\sigma - 1) = 0$ and $\mu (b - q) = 0$. Finally, by induction, we obtain $x^{m}  \triangleleft H = \sigma^{m} x^{m}$ which is again trivially fulfilled since $x^{m} = 0$.

We move on to checking \equref{mp2}. Again by induction we obtain $h \triangleright H^{n} = H^{n}$ and $x \triangleright H^{n} = 0$, equalities trivially fulfilled since $H^{n} = 1$.  Now from the following equalities:
\begin{eqnarray*}
h  \triangleright XH &=& a H - a H^{2} + b \overline{q} HX\\
h  \triangleright \overline{q} HX &=&  \overline{q}  a H - \overline{q}  a H^{2} + b \overline{q} HX
\end{eqnarray*}
we obtain $\overline{q}  a = a$ and since $\overline{q}  \neq 1$ we get $a = 0$. Finally, the following equalities:
\begin{eqnarray*}
x  \triangleright XH &=& \alpha H - \alpha H^{2}\\
x  \triangleright \overline{q} HX &=&  \overline{q}  \alpha \sigma H - \overline{q} \alpha \sigma H^{2}
\end{eqnarray*}
yield $\alpha (1 - \sigma \overline{q}) = 0$.

We arrived at the last compatibility condition, namely \equref{mp4}. For the pairs $(h,\, H)$, $(h,\, X)$ and $(x, \, H)$ the aforementioned compatibility is trivially fulfilled. We are left to check \equref{mp4} for the pair $(x, \, X)$. In this case it comes down to the following:
\begin{eqnarray*}
x \triangleleft X \ot h \triangleright H + x \triangleleft 1 \ot h \triangleright X + 1 \triangleleft X \ot x \triangleright H + 1 \triangleleft 1 \ot x \triangleright X = \\
h \triangleleft H \ot x \triangleright X + h \triangleleft X \ot x \triangleright 1 + x \triangleleft H \ot 1 \triangleright X + x \triangleleft X \ot 1 \triangleright 1
\end{eqnarray*}
which is equivalent to:
\begin{eqnarray*}
\mu (1 - h) \ot H + x \ot bX + 1 \ot \alpha (1 - H) = h \ot \alpha (1 - H) + \sigma x \ot X + \mu (1 - h) \ot 1.
\end{eqnarray*}
Therefore, we obtain $b = \sigma$ and $\alpha = \mu$. Putting all together, we have obtained the following matched pairs:
\begin{eqnarray*}
&& h \triangleright H = H,\,\, x \triangleright H = 0,\,\, x \triangleright X = \alpha(1 - H),\,\, h \triangleright X = \sigma X\\
&& h  \triangleleft  H = h,\,\, h \triangleleft X = 0,\,\, x  \triangleleft X = \alpha (1 - h),\,\, x  \triangleleft H = \sigma x
\end{eqnarray*}
where $\sigma^{n} = 1 = \sigma^{m}$, $\alpha (1 - \overline{q} \sigma) = 0$, $\alpha (\sigma - q) = 0$. 

We distinguish two cases based on whether $\overline{q} =  q^{-1}$ or $\overline{q} \neq q^{-1}$:

$1)$ If $\overline{q} \neq  q^{-1}$ then we can not have both $(\sigma - q) = 0$ and $(1 - \overline{q} \sigma) = 0$ so we obtain $\alpha = 0$. This gives rise to the first matched pair in the statement of our result. 

$2)$ Assume now that $\overline{q} =  q^{-1}$. To start with, as $\overline{q} \in P_{n}(k)$ and $ q^{-1} = q^{m-1} \in P_{m}(k)$ we obtain $\overline{q} \in P_{n}(k) \bigcap P_{m}(k)$ which is only possible if $m=n$. Now, if $\alpha = 0$ we obtain again our first matched pair while if $\alpha \neq 0$ we must have $\sigma =  \overline{q}  = q^{-1}$ and this gives rise to the second matched pair.
Although initially we only imposed the compatibility conditions defining a matched pair to be fulfilled for the generators of the two Hopf algebras, it can be easily seen now that they are fulfilled for all elements of $\mathbb{T}_{n^{2}}(\overline{q})$ and respectively $T_{m^{2}}(q)$. For example, in the case of the second matched pair whose formula is more complicated, we have:
\begin{eqnarray*}
x^{t} \triangleleft X &=& \alpha \bigl(q^{t-1} + q^{t-2} + \cdots + 1\bigl) x^{t-1} - \alpha q^{t-1} \bigl(q^{t-1} + q^{t-2} + \cdots + 1\bigl) hx^{t-1}\\
x \triangleright X^{t} &=& \alpha \bigl(q^{t-1} + q^{t-2} + \cdots + 1\bigl) X^{t-1} - \alpha q^{t-1} \bigl(q^{t-1} + q^{t-2} + \cdots + 1\bigl) HX^{t-1}
\end{eqnarray*}
for all $t \in \{1, 2, \cdots n\}$. Moreover, one can see that these formulae respect the relations $x^{n} = 0$ and $X^{n} =0$ from the two Taft algebras just as
a simple consequence of the identity $q^{n-1} + q^{n-2} + \cdots + 1 = 0$.
\end{proof}

Having listed all possible matched pairs between two Taft algebras
in \thref{main} the next step in our approach is to describe the
corresponding bicrossed products.

\begin{corollary}\colabel{descriere}
Let $m$, $n \in \NN^{*} \setminus \{1\}$, $d = (m,\, n)$ and $q \in P_{m}(k)$, $\overline{q} \in P_{n}(k)$.

\begin{enumerate}
\item[1)] If $\overline{q} \neq q^{n-1}$ then a Hopf algebra $E$ factors through $\mathbb{T}_{n^{2}}(\overline{q})$ and $T_{m^{2}}(q)$ if and only if $E \cong \mathrm{T}_{n,\,m}^{\, \sigma}(\overline{q},\,q)$ for some $\sigma \in U_{d}(k)$, where we denote by $\mathrm{T}_{n,\,m}^{\, \sigma}(\overline{q},\,q)$ the Hopf algebra generated by $h$, $x$, $H$ and $X$ subject to the relations:
\begin{eqnarray*}
h^{m} = H^{n} = 1,&&  x^{m} = X^{n} = 0,\,\,\,\,\,\,\,\, xh = q hx, \,\,\,\,\,\,\,\, XH = \overline{q}HX \\
xH = \sigma Hx,&& hX = \sigma Xh,\,\,\,\,\,\,\,\, xX = \sigma Xx,\,\,\,\,\,\,\,\, hH = Hh
\end{eqnarray*}
such that $h$, $H$ are group-like elements, $x$ is an $(h,1)$-primitive element, $X$ is an $(H,1)$-primitive element and the antipode is given by:
\begin{eqnarray*}
S(h) = h^{m-1},\,\,\, S(H) = H^{n-1}, \,\,\, S(x) = -q^{m-1}h^{m-1}x, \,\,\, S(X) = - \overline{q}^{n-1} H^{n-1}X.
\end{eqnarray*}
\item[2)] If $\overline{q} = q^{n-1}$ then a Hopf algebra $E$ factors through $\mathbb{T}_{n^{2}}(q^{n-1})$ and $T_{n^{2}}(q)$ if and only if $E \cong \mathrm{T}_{n,\,n}^{\, \sigma}(q^{n-1},\,q)$ for some $\sigma \in U_{n}(k)$ or $E \cong \mathrm{Q}_{n}^{\alpha}(q)$ for some $\alpha \in k^{*}$ where we denote by $\mathrm{Q}_{n}^{\alpha}(q)$ the Hopf algebra generated by $h$, $x$, $H$ and $X$ subject to the relations:
\begin{eqnarray*}
h^{n} = H^{n} = 1,\,\,\,\,\,\,\,\, x^{n} = X^{n} = 0,\,\,\,\,\,\,\,\, xh = q hx, \,\,\,\,\,\,\,\, XH = q^{n-1}HX \\
xH = q Hx,\,\,\,\, hX = q Xh,\,\,\,\, xX - qXx = \alpha (1-Hh),\,\,\,\, hH = Hh
\end{eqnarray*}
such that $h$, $H$ are group-like elements, $x$ is an $(h,1)$-primitive element, $X$ is an $(H,1)$-primitive element and the antipode is given by:
\begin{eqnarray*}
S(h) = h^{n-1},\,\,\, S(H) = H^{n-1}, \,\,\, S(x) = -q^{n-1}h^{n-1}x, \,\,\, S(X) = - q H^{n-1}X.
\end{eqnarray*}
\end{enumerate}
\end{corollary}

\begin{proof}
The proof relies on \thref{carMaj} and \thref{main}. Suppose first that $\overline{q} \neq q^{n-1}$. We will see that $\mathrm{T}_{n,\,m}^{\, \sigma}(\overline{q},\,q)$ is the bicrossed product $\mathbb{T}_{n^{2}}(\overline{q}) \bowtie T_{m^{2}}(q)$ corresponding to the matched pair described in the first part of \thref{main}. Indeed, up to canonical identification, $\mathbb{T}_{n^{2}}(\overline{q}) \bowtie T_{m^{2}}(q)$ is generated as an algebra by $X = X \bowtie 1$, $H = H \bowtie 1$, $x = 1 \bowtie x$ and $h = 1 \bowtie h$. Then, we have:
\begin{eqnarray*}
xX &=& (1 \bowtie x) (X \bowtie 1) = x_{(1)}  \triangleright X_{(1)} \bowtie x_{(2)} \triangleleft X_{(2)}\\
&=& x  \triangleright X \bowtie h \triangleleft H + x  \triangleright 1 \bowtie h \triangleleft X + 1 \triangleright X \bowtie x \triangleleft H + 1 \triangleright 1 \bowtie x \triangleleft X\\
&=& X \bowtie \sigma x = \sigma Xx.
\end{eqnarray*}
In the same manner it can be proved that $hH = Hh$, $xH = \sigma Hx$ and $hX = \sigma Xh$.

Assume now that $\overline{q} = q^{n-1}$. Then, according to \thref{main} we have two families of matched pairs. The bicrossed product corresponding to the first family is just the one described above. The proof will be finished once we prove that the bicrossed product corresponding to the matched pair in \thref{main}, 2) coincides with $\mathrm{Q}_{n}^{\alpha}(q)$. Indeed, in this case we have:
\begin{eqnarray*}
xX &=& (1 \bowtie x) (X \bowtie 1) = x_{(1)}  \triangleright X_{(1)} \bowtie x_{(2)} \triangleleft X_{(2)}\\
&=& x  \triangleright X \bowtie h \triangleleft H + x  \triangleright 1 \bowtie h \triangleleft X + 1 \triangleright X \bowtie x \triangleleft H + 1 \triangleright 1 \bowtie x \triangleleft X\\
&=& \alpha (1 - H) \bowtie h + X \bowtie qx + 1 \bowtie \alpha(1-h)\\
&=& qXx + \alpha(1-Hh)\\
xH &=& (1 \bowtie x)(H \bowtie 1) =  x_{(1)}  \triangleright H \bowtie x_{(2)} \triangleleft H\\
 &=& x  \triangleright H \bowtie h \triangleleft H + 1  \triangleright H \bowtie x \triangleleft H = H \bowtie qx\\
 &=& qHx.
\end{eqnarray*}
Similarly, one can see that $hH = Hh$ and $hX = qXh$ and the proof is now complete.
\end{proof}

We are now in a position to prove our first result on the classification of Hopf algebras which factors through two Taft algebras.

\begin{theorem}\thlabel{clasificare}
With the above notations we have:
\begin{enumerate}
\item[1)] If $n \neq m$ or $m=n$ and $\overline{q} \neq q$ then there exists an isomorphism of Hopf algebras $\mathrm{T}_{n,\,m}^{\, \sigma}(\overline{q},\,q) \cong \mathrm{T}_{n,\,m}^{\, \sigma '}(\overline{q},\,q)$ if and only if $\sigma = \sigma '$;\\
\item[2)] If $m=n$ and $\overline{q} = q$ then there exists an isomorphism of Hopf algebras $\mathrm{T}_{n,\,n}^{\, \sigma}(q,\, q) \cong \mathrm{T}_{n,\,n}^{\, \sigma '}(q,\, q)$ if and only if $\sigma ' = \sigma$ or $\sigma ' = \sigma^{-1}$;\\
\item[3)] There exists an isomorphism of Hopf algebras $\mathrm{Q}_{n}^{\alpha}(q) \cong \mathrm{Q}_{n}^{1}(q)$ for all $\alpha \in k^{*}$. Moreover, $\mathrm{Q}_{n}^{1}(q)$ is isomorphic to the Drinfeld double $D\bigl(T_{n^{2}}(q)\bigl)$;\\
\item[4)] $\mathrm{Q}_{n}^{1}(q)$ is not isomorphic to any of the
Hopf algebras $\mathrm{T}_{n,\,n}^{\, \sigma}(q^{n-1},\, q)$.
\end{enumerate}
\end{theorem}

\begin{proof}
1) and 2). We will prove the first two statements together. Suppose that
$\mathrm{T}_{n,\,m}^{\, \sigma}(\overline{q},\,q)$ and
resppectively $\mathrm{T}_{n,\,m}^{\, \sigma '}(\overline{q},\,q)$
are the bicrossed products corresponding to the matched pairs
$(\mathbb{T}_{n^{2}}(\overline{q}), \, T_{m^{2}}(q),\,
\triangleright, \, \triangleleft)$ and
$(\mathbb{T}_{n^{2}}(\overline{q}), \, T_{m^{2}}(q),\,
\triangleright ', \, \triangleleft ')$ implemented by $\sigma$ and
$ \sigma '$ as in \thref{main}, 1). As mentioned before, our
approach relies on \thref{toatemorf}: any Hopf algebra map $\psi:
\mathrm{T}_{n,\,m}^{\, \sigma}(\overline{q},\,q) \to
\mathrm{T}_{n,\,m}^{\, \sigma '}(\overline{q},\,q)$ is
parameterized by a quadruple $(u, p, r, v)$ consisting of unital
coalgebra maps $u: \mathbb{T}_{n^{2}}(\overline{q}) \to
\mathbb{T}_{n^{2}}(\overline{q})$, $p:
\mathbb{T}_{n^{2}}(\overline{q}) \to T_{m^{2}}(q)$, $r:
T_{m^{2}}(q) \to \mathbb{T}_{n^{2}}(\overline{q})$, $v:
T_{m^{2}}(q) \to T_{m^{2}}(q)$ satisfying the compatibility
conditions \equref{C1}-\equref{C8}. To start with, since all four
maps above are unital coalgebra maps, \leref{primitive} gives:
\begin{eqnarray*}
&& u(1) = 1,\,\, p(1) = 1,\,\, r(1) = 1,\,\, v(1) = 1\\
&& u(H) = H^{a},\,\, p(H) = h^{b},\,\, r(h) = H^{c},\,\, v(h) = h^{d}\\
&& u(X) \in \mathcal{P}_{H^{a},\, 1} \bigl(\mathbb{T}_{n^{2}}(\overline{q}) \bigl),\,\, p(X) \in \mathcal{P}_{h^{b},\, 1} \bigl(T_{m^{2}}(q)\bigl)\\
&& r(x) \in \mathcal{P}_{H^{c},\, 1} \bigl(\mathbb{T}_{n^{2}}(\overline{q}) \bigl),\,\, v(x) \in \mathcal{P}_{h^{d},\, 1} \bigl(T_{m^{2}}(q)\bigl).
\end{eqnarray*}
Applying \equref{C1} for $a = X$ yields:
\begin{equation}\eqlabel{22}
u(X) \ot h^{b} + 1 \ot p(X) = H^{a} \ot p(X) + u(X) \ot 1.
\end{equation}
Putting together  \equref{22} and the fact that $u(X) \in \mathcal{P}_{H^{a},\, 1} \bigl(\mathbb{T}_{n^{2}}(\overline{q}) \bigl)$ and respectively $p(X) \in \mathcal{P}_{h^{b},\, 1} \bigl(T_{m^{2}}(q)\bigl)$ we obtain the following possibilities for the pair of maps $(u,\,p)$:
\begin{eqnarray*}
&{\rm I.}& u(H) = H,\,\, u(X) = \alpha (1 - H),\,\, p(H) = h,\,\, p(X) =  \alpha (1 - h),\,\, \alpha \in k;\\
&{\rm II.}& u(H) = H,\,\, u(X) = \alpha (1 - H) + \beta X,\,\, p(H) = 1,\,\, p(X) =  0,\,\, \alpha, \beta \in k;\\
&{\rm III.}& u(H) = H,\,\, u(X) = \alpha (1 - H),\,\, p(H) = h^{b},\,\, p(X) =  \alpha (1 - h^{b}),\,\, \alpha \in k,\, b \in \NN^{*}\setminus\{1\};\\
&{\rm IV.}& u(H) = 1,\,\, u(X) = 0,\,\, p(H) = h,\,\, p(X) =  \alpha (1 - h) + \beta x,\,\, \alpha, \beta \in k;\\
&{\rm V.}& u(H) = H^{a},\,\, u(X) = \alpha(1 - H^{a}),\,\, p(H) = h,\,\, p(X) =  \alpha (1 - h),\,\, \alpha \in k,\, a \in \NN^{*}\setminus\{1\};\\
&{\rm VI.}& u(H) = 1,\,\, u(X) = 0,\,\, p(H) = 1,\,\, p(X) =  0;\\
&{\rm VII.}& u(H) = 1,\,\, u(X) = 0,\,\, p(H) = h^{b},\,\, p(X) =  \alpha (1 - h^{b}),\,\, \alpha \in k,\, b \in \NN^{*}\setminus\{1\};\\
&{\rm VIII.}& u(H) = H^{a},\,\, u(X) = \alpha(1 - H^{a}),\,\, p(H) = 1,\,\, p(X) =  0,\,\, \alpha \in k,\, a \in \NN^{*}\setminus\{1\};\\
&{\rm IX.}& u(H) = H^{a},\,\, u(X) = \alpha (1 - H^{a}),\,\, p(H) = h^{b},\,\, p(X) =  \alpha (1 - h^{b}),\\
&& \alpha \in k,\,\,a, b \in \NN^{*}\setminus\{1\}.\\
\end{eqnarray*}
Now by applying \equref{C3} for $(a,\,b) = (X,\, H)$ and respectively $(a,\, b) = (H,\,X)$ we get:
\begin{eqnarray*}
u(XH) &=& u(X) H^{a} + p(X) \triangleright ' H^{a}\\
u(HX) &=& H^{a} \bigl(h^{b} \triangleright  ' u(X)\bigl).
\end{eqnarray*}
As $XH = \overline{q} HX$ we obtain:
\begin{equation}\eqlabel{co1}
u(X) H^{a} + p(X) \triangleright ' H^{a} = \overline{q} H^{a} \bigl(h^{b} \triangleright  ' u(X)\bigl).
\end{equation}
Analogously, using \equref{C4} in the same manner as above yields:
\begin{equation}\eqlabel{co2}
\bigl(p(X)  \triangleleft ' H^{a} \bigl) h^{b} = \overline{q} \bigl(h^{b} \triangleleft ' u(X) \bigl) h^{b} + \overline{q} h^{b} p(X).
\end{equation}
Suppose the maps $u$, $p$ are given as in I. Then the compatibility condition \equref{co1} gives:
\begin{eqnarray*}
\alpha (1 - H) H + \alpha (1 - h) \triangleright ' H = \overline{q} H \bigl(h \triangleright ' \alpha (1 - H)\bigl)
\end{eqnarray*}
which comes down to $\alpha = \alpha \overline{q}$ and since $\overline{q} \neq 1$ we get $\alpha = 0$. Therefore $u(X) = p(X) = 0$. This implies:
\begin{eqnarray*}
\psi(X \bowtie 1) = u(X_{(1)}) \bowtie p(X_{(2)}) = u(X) \bowtie p(H) + u(1) \bowtie p(X) = 0
\end{eqnarray*}
so in this case $\psi$ is not an isomorphism. Using the same strategy as above, a rather long but straightforward computation based on the compatibility conditions \equref{co1} and \equref{co2} rules out the other cases as well except for the second one and the fourth one if $m=n$ and $q = \overline{q}$.
Indeed, suppose the maps $u$, $p$ are given as in II. Then the compatibility condition \equref{co2} is trivially fulfilled while \equref{co1} gives:
\begin{eqnarray*}
\bigl(\alpha (1 - H) H + \beta X\bigl) H = \overline{q} H \bigl(\alpha (1 - H) + \beta X\bigl)
\end{eqnarray*}
which implies $\alpha = \alpha \overline{q}$ and since
$\overline{q} \neq 1$ we get $\alpha = 0$. Thus $u(X) = \beta X$
and $p(H) = 1$. Furthermore, using \equref{C3} and \equref{C4} we
obtain the following general formulas:
\begin{eqnarray*}
&& u(X^{j} H^{i}) = \beta^{j}\, X^{j}\, H^{i},\,\, {\rm where}\,\, \beta \in k,\,\, i,\,j \in \{0,\,1,\, \cdots,\, n-1\}\\
&& p(Y) = \epsilon(Y) 1,\,\, {\rm for} \,\, {\rm all} \,\, Y \in \mathbb{T}_{n^{2}}(\overline{q}).
\end{eqnarray*}
Finally, we discuss the fourth case, i.e. the maps $u$, $p$ are given as in IV. Then \equref{co1} is trivially fulfilled while \equref{co2} comes down to:
\begin{eqnarray*}
\alpha h - \alpha h^{2} + \beta xh = \overline{q} \alpha h - \overline{q} \alpha h^{2} + \beta \overline{q} hx
\end{eqnarray*}
so we get $\alpha = 0$ and $q = \overline{q}$. As we noticed before, $q = \overline{q}$ also implies $m = n$. Thus, if $m = n$ and $q = \overline{q}$ we can also have:
\begin{eqnarray*}
&& u(Y) = \epsilon(Y)1,\,\, {\rm for} \,\, {\rm all} \,\, Y \in \mathbb{T}_{n^{2}}(\overline{q})\\
&& p(X^{i}H^{j}) = \zeta^{i} X^{i}H^{j},\,\, {\rm where}\,\, \zeta \in k,\,\, i,\,j \in \{0,\,1,\, \cdots,\, n-1\}.
\end{eqnarray*}

Next we focus on the maps $r$ and $v$. To this end, we apply \equref{C2} for $g = X$ which yields:
\begin{eqnarray*}
r(x) \ot h^{d} + 1 \ot v(x) = H^{c} \ot v(x) + r(x) \ot 1.
\end{eqnarray*}
We have $r(x) \in \mathcal{P}_{H^{c},\, 1} \bigl(\mathbb{T}_{n^{2}}(\overline{q}) \bigl)$ and $v(x) \in \mathcal{P}_{h^{d},\, 1} \bigl(T_{m^{2}}(q)\bigl)$ and we proceed as in the case of the maps $u$ and $p$ but using this time the compatibility conditions \equref{C5} and \equref{C6}. We obtain:
\begin{eqnarray*}
&& r(y) = \epsilon(y) 1,\,\, {\rm for} \,\, {\rm all} \,\, y \in T_{m^{2}}(q)\\
&& v(x^{l} h^{t}) = \eta^{l}\, x^{l}\, h^{t},\,\, {\rm where}\,\, \eta \in k,\,\, l,\,t \in \{0,\,1,\, \cdots,\, m-1\}.
\end{eqnarray*}
In addition, if $q = \overline{q}$ (and thus $m = n$) we also have
the following pair of maps $(r,\,v)$:
\begin{eqnarray*}
&& r(x^{l} h^{t}) = \gamma^{l} x^{l} h^{t},\,\, {\rm where}\,\, \gamma \in k,\,\, l,\,t \in \{0,\,1,\, \cdots,\, m-1\}\\
&& v(y) = \epsilon(y) 1,\,\, {\rm for} \,\, {\rm all} \,\, y \in T_{m^{2}}(q).
\end{eqnarray*}

Suppose first that $m \neq n$ or $m = n$ and $q \neq \overline{q}$. Then, the only quadruple $(u,\, p,\, r,\, v)$ which might satisfy the compatibility conditions \equref{C1}-\equref{C8} is given as follows:
\begin{eqnarray*}
&& u(X^{j} H^{i}) = \beta^{j}\, X^{j}\, H^{i},\,\, {\rm where}\,\, \beta \in k,\,\, i,\,j \in \{0,\,1,\, \cdots,\, n-1\}\\
&& p(Y) = \epsilon(Y) 1,\,\, {\rm for} \,\, {\rm all} \,\, Y \in \mathbb{T}_{n^{2}}(\overline{q})\\
&& r(y) = \epsilon(y) 1,\,\, {\rm for} \,\, {\rm all} \,\, y \in T_{m^{2}}(q)\\
&& v(x^{l} h^{t}) = \eta^{l}\, x^{l}\, h^{t},\,\, {\rm where}\,\, \eta \in k,\,\, l,\,t \in \{0,\,1,\, \cdots,\, m-1\}
\end{eqnarray*}
and the corresponding Hopf algebra morphism $\psi:
\mathrm{T}_{n,\,m}^{\, \sigma}(\overline{q},\,q) \to
\mathrm{T}_{n,\,m}^{\, \sigma '}(\overline{q},\,q)$ is given by
$\psi(a \bowtie g) = u(a) \bowtie v(g)$ for any $a \in
\mathbb{T}_{n^{2}}(\overline{q})$ and $g \in T_{m^{2}}(q)$. First
notice that since $\psi (X \bowtie 1) = \beta X \bowtie 1$ and
$\psi(1 \bowtie x) = 1 \bowtie \eta x$ and we want $\psi$ to be an
isomorphism we need to have $\beta$, $\eta \in k^{*}$.

Now considering the way we obtained the maps $u$, $p$, $r$, $v$
listed above it is obvious that the compatibility conditions
\equref{C1}-\equref{C6} are fulfilled. We are left to check the
compatibility conditions \equref{C7} and \equref{C8} which in this
case boil down to:
\begin{eqnarray}
v(g) \triangleright ' u(b) = u(g \triangleright b) \eqlabel{e1}\\
v(g) \triangleleft ' u(b) = v(g \triangleleft b)\eqlabel{e2}
\end{eqnarray}
for all $g \in T_{m^{2}}(q)$, $b \in
\mathbb{T}_{n^{2}}(\overline{q})$. Now considering $(g,\, b) :=
(h,\, X)$ in \equref{e1} we obtain $\beta \sigma ' X = \beta
\sigma X$ and since $\beta \neq 0$ we must have $\sigma = \sigma '$.
To conclude, in this case, $\mathrm{T}_{n,\,m}^{\, \sigma}(\overline{q},\,q)$ is isomorphic to $\mathrm{T}_{n,\,m}^{\, \sigma
'}(\overline{q},\,q)$ if and only if $\sigma = \sigma '$.

Assume now that  $m = n$ and $q = \overline{q}$. Then, besides the
quadruple listed above which implies that $\mathrm{T}_{n,\,n}^{\,
\sigma}(q,\, q) \cong \mathrm{T}_{n,\,n}^{\, \sigma '}(q,\, q)$ if
and only if $\sigma = \sigma '$, we can also have the following
possibilities:

\begin{eqnarray*}
&{\rm I.}& u_{1}(Y) = \epsilon(Y)1,\,\, p_{1}(X^{i}H^{j}) = \zeta^{i} x^{i}h^{j},\,\, r_{1}(x^{l} h^{t}) = \gamma^{l} X^{l} H^{t},\,\, v_{1}(y) = \epsilon(y) 1;\\
&{\rm II.}& u_{2}(Y) = \epsilon(Y)1,\,\, p_{2}(X^{i}H^{j}) = \zeta^{i} x^{i}h^{j},\,\, r_{2}(y) = \epsilon(y) 1,\,\, v_{2}(x^{l} h^{t}) = \eta^{l} x^{l} h^{t};\\
&{\rm III.}& u_{3}(X^{j} H^{i}) = \beta^{j}\, X^{j}\, H^{i},\,\, p_{3}(Y) = \epsilon(Y)1,\,\, r_{3}(x^{l} h^{t}) = \gamma^{l} X^{l} H^{t},\,\, v_{3}(y) = \epsilon(y) 1;\\
&& {\rm where}\,\, i,\,j \in \{0,\,1,\, \cdots,\, n-1\},\,\, l,\,t \in \{0,\,1,\, \cdots,\, m-1\},\,\, \gamma,\, \zeta \in k^{*},\\
&& Y \in \mathbb{T}_{n^{2}}(\overline{q}),\,\, y \in T_{n^{2}}(q).
\end{eqnarray*}

Suppose the quadruple $(u,\, p,\, r,\, v)$ is given as in I. Then the compatibility condition \equref{C7} writes:
\begin{eqnarray*}
r_{1}(g) \epsilon(b) = p_{1} \bigl(g_{(1)} \triangleright b_{(1)} \bigl) \triangleright ' r_{1} \bigl(g_{(2)} \triangleleft b_{(2)}\bigl)
\end{eqnarray*}
for all $g \in T_{m^{2}}(q)$, $b \in
\mathbb{T}_{n^{2}}(\overline{q})$. Setting $(b,\, g) = (H, \, x)$
in the above equality yields $\gamma X = \gamma \sigma \sigma '
X$. If $\gamma = 0$ then we get $r_{1}(x) = 0$ and thus
\begin{eqnarray*}
\psi(1 \bowtie x) = r_{1}(x_{(1)}) \bowtie' v_{1}(x_{(2)}) = r_{1}(x)  \bowtie' 1 = 0.
\end{eqnarray*}
Therefore, $\psi$ is not bijective unless $\gamma \neq 0$. Hence
for $\mathrm{T}_{n,\,n}^{\, \sigma}(q,\, q)$ and
$\mathrm{T}_{n,\,n}^{\, \sigma '}(q,\, q)$ to be isomorphic we
need to have $\sigma \sigma ' = 1$. Furthermore, it can be easily
seen by a straightforward computation that \equref{C8} is also
fulfilled and $\psi$ is a Hopf algebra isomorphism if and only if
$\sigma \sigma ' = 1$.

If the quadruple $(u,\, p,\, r,\, v)$ is given as in II or III
then the corresponding Hopf algebra map $\psi$ is not an
isomorphism. Indeed, assume first that $(u,\, p,\, r,\, v)$ is
given as in II. Then, for all $a \in
\mathbb{T}_{n^{2}}(\overline{q})$ and $g \in T_{n^{2}}(q)$ we
have:
\begin{eqnarray*}
\psi(a \bowtie g) = \epsilon(a_{(1)})  \bigl( p(a_{(2)})
\triangleright' \epsilon(g_{(1)}) \bigl) \,\, \bowtie' \, \bigl(
p(a_{(3)}) \triangleleft' \epsilon(g_{(2)}) \bigl) \, v(g_{(3)}) = 1 \, \bowtie' \, p(a) v(g)
\end{eqnarray*}
and it is obvious that $\psi$ is not surjective. Similarly, if the quadruple $(u,\, p,\, r,\, v)$ is given as in III, then for all $a \in \mathbb{T}_{n^{2}}(\overline{q})$ and $g \in T_{n^{2}}(q)$ we have:
\begin{eqnarray*}
\psi(a \bowtie g) = u(a_{(1)}) \, \bigl( \epsilon(a_{(2)})
\triangleright' r(g_{(1)}) \bigl) \,\, \bowtie' \, \bigl(
\epsilon(a_{(3)}) \triangleleft' r(g_{(2)}) \bigl) \, \epsilon(g_{(3)}) = u(a) r(g) \, \bowtie' \, 1
\end{eqnarray*}
and as before $\psi$ is again not surjective.

3) Indeed, it is straightforward to see that the $k$-linear map $\varphi : \mathrm{Q}_{n}^{\alpha}(q) \to \mathrm{Q}_{n}^{1}(q)$ defined by:
\begin{eqnarray*}
\varphi(H) = H,\,\, \varphi(h) = h,\,\, \varphi(x) = x,\,\, \varphi(X) = \alpha^{-1}X
\end{eqnarray*}
is an isomorphism of Hopf algebras for all $\alpha \in k^{*}$.

For the second part of the statement, recall that the Drinfel'd
double $D\bigl(T_{n^{2}}(q)\bigl)$ factors through
$\bigl(T_{n^{2}}(q)^{\ast}\bigl)^{cop}$ and $T_{n^{2}}(q)$.
Putting together the two Hopf algebra isomorphisms
$T_{m^{2}}(q)^{\ast} \cong T_{m^{2}}(q)$ and $T_{m^{2}}(q)^{cop}
\cong T_{m^{2}}(q^{n-1})$ noted in \nameref{prel} we can conclude
that the Hopf algebras $\bigl(T_{n^{2}}(q)^{\ast}\bigl)^{cop}$ and
respectively $T_{n^{2}}(q^{n-1})$ are also isomorphic. Therefore,
the Drinfel'd double $D\bigl(T_{n^{2}}(q)\bigl)$ is one of the
bicrossed products between $T_{n^{2}}(q^{n-1})$ and
$T_{n^{2}}(q)$. To this end, we compute:
\begin{eqnarray*}
x \triangleleft X^{\ast}  &\stackrel{\equref{dubluact}}{=}&
X^{\ast} \bigl(S^{-1}(h) x\bigl)h + X^{\ast}
\bigl(S^{-1}(h)\bigl)x +
X^{\ast} \bigl(S^{-1}(x)\bigl)\\
&{=}& X^{\ast} (h^{n-1} x) h + X^{\ast}(-h^{n-1} x)
\stackrel{\equref{dual}}{=} h-1.
\end{eqnarray*}
We can now conclude that the matched pair which gives the
Drinfel'd double coincides with the matched pair depicted in
\thref{main}, 2) for $\alpha = -1$. Since we proved that the
corresponding matched pair, namely $\mathrm{Q}_{n}^{-1}(q)$, is
isomorphic to $\mathrm{Q}_{n}^{1}(q)$ the conclusion follows.

4) We use again \thref{toatemorf}. Suppose there exists an isomorphism $\psi$ between $\mathrm{Q}_{n}^{1}(q) = \mathbb{T}_{n^{2}}(q^{n-1}) \bowtie \mathbb{T}_{n^{2}}(q)$ and $\mathrm{T}_{n,\,n}^{\, \sigma}(q,\, q) = \mathbb{T}_{n^{2}}(q^{n-1}) \bowtie '  \mathbb{T}_{n^{2}}(q)$. Based on the first part of the proof it follows that $\psi$ is determined by the following quadruple of coalgebra maps:
\begin{eqnarray*}
&& u(X^{j} H^{i}) = \beta^{j}\, X^{j}\, H^{i},\,\, {\rm where}\,\, \beta \in k,\,\, i,\,j \in \{0,\,1,\, \cdots,\, n-1\}\\
&& p(Y) = \epsilon(Y) 1,\,\, {\rm for} \,\, {\rm all} \,\, Y \in \mathbb{T}_{n^{2}}(\overline{q})\\
&& r(y) = \epsilon(y) 1,\,\, {\rm for} \,\, {\rm all} \,\, y \in T_{m^{2}}(q)\\
&& v(x^{l} h^{t}) = \eta^{l}\, x^{l}\, h^{t},\,\, {\rm where}\,\, \eta \in k,\,\, l,\,t \in \{0,\,1,\, \cdots,\, m-1\}.
\end{eqnarray*}
Remark that the compatibilities \equref{C1} - \equref{C6}
are fulfilled. Now, in this case the compatibility condition
\equref{C7} comes down to $v(g) \triangleright ' u(b) = u(g
\triangleright b)$ which evaluated at $(g,\,b) = (x,\, X)$ gives
$\eta \beta x \triangleright ' X = u(1 - H)$. Therefore $1 = H$
and we have reached a contradiction.
\end{proof}

\begin{remark}
If $m=n=2$ and $q = \overline{q} = -1$ then the two Taft algebras
we considered coincide with the Sweedler's Hopf algebra. In this
case, \thref{clasificare} provides three isomorphism classes for
the Hopf algebras which factors through two Sweedler's Hopf
algebras, namely $\mathrm{T}_{2,\,2}^{1}(-1,\, -1)$,
$\mathrm{T}_{2,\,2}^{-1}(-1,\, -1)$ and respectively
$Q_{2}^{1}(-1)$. After carefully comparing \coref{descriere} and
\cite[Theorem 2.4]{Gabi} we obtain the following:
\begin{eqnarray*}
\mathrm{T}_{2,\,2}^{1}(-1,\, -1) = \mathbb{H}_{4}  \otimes
H_{4},\,\, \mathrm{T}_{2,\,2}^{-1}(-1,\, -1) = \mathcal{H}_{16,
0},\,\, Q_{2}^{1}(-1) =\mathcal{H}_{16, 1}.
\end{eqnarray*}
\end{remark}

In fact, \thref{clasificare} allows us to compute the number of
types of Hopf algebras which factor through two Taft algebras.

\begin{corollary}\colabel{nr}
Let $m$, $n \in \NN^{*} \setminus \{1\}$, $d = (m,\,n)$, $\nu(d) = |U_{d}(k)|$ and
consider $\sharp^{n,\,m}_{\overline{q},\,q}$ to be the number of
types of Hopf algebras which factor through
$\mathbb{T}_{n^{2}}(\overline{q})$ and $T_{m^{2}}(q)$. Then, we
have:
\begin{eqnarray*}
\sharp^{n,\,m}_{\overline{q},\,q} = \left\{\begin{array}{rcl} 3, &
\mbox{if}& m = n = 2\\ \nu(d) & \mbox{if}& m \neq n \,\,
\mbox{or} \,\, m = n \geqslant 3 \,\, \mbox{and} \,\, \overline{q}
\notin \{q,\, q^{n-1}\}\\ \frac{\nu(d)}{2} + 1& \mbox{if}& m =
n\geqslant 3, \,\, \overline{q} = q\,\,  \mbox{and}\,\, \nu(d)
\,\, \mbox{is}\,\, \mbox{even}\\ \frac{\nu(d)+1}{2}& \mbox{if}& m
= n\geqslant 3, \,\, \overline{q} = q\,\,  \mbox{and}\,\, \nu(d)
\,\, \mbox{is}\,\, \mbox{odd}\\  \nu(d)+1 & \mbox{if}& m = n
\geqslant 3, \,\, \overline{q} = q^{n-1} \end{array}\right.
\end{eqnarray*}
\end{corollary}

\begin{proof}
Suppose first that $ m \neq n$ or  $m = n \geqslant 3$ and
$\overline{q} \notin \{q,\, q^{n-1}\}$. Then according to
\thref{clasificare} two Hopf algebras $\mathrm{T}_{n,\,m}^{\,
\sigma}(\overline{q},\,q)$ and respectively
$\mathrm{T}_{n,\,m}^{\, \sigma '}(\overline{q},\,q)$, where
$\sigma$, $\sigma ' \in U_{d}(k)$, are isomorphic if and only if
$\sigma = \sigma ' $ and the conclusion follows. Now if $m = n
\geqslant 3$ and $\overline{q} = q$ then $\mathrm{T}_{n,\,n}^{\,
\sigma}(q,\, q)$ and respectively $\mathrm{T}_{n,\,n}^{\, \sigma
'}(q,\, q)$, where $\sigma$, $\sigma ' \in U_{d}(k)$, are
isomorphic if and only if $\sigma = \sigma ' $ or $\sigma \sigma '
=1$. This leads to a discussion on whether $\nu(d)$ is odd or even
and the two formulas in our statement can be easily derived.
Finally, let $m=n \geqslant 3 $ and  $\overline{q} = q^{n-1}$. In
this case, the Hopf algebras which factor through
$\mathbb{T}_{n^{2}}(q^{n-1})$ and $T_{n^{2}}(q)$ are
$\mathrm{Q}_{n}^{1}(q)$  together with $\mathrm{T}_{n,\,n}^{\,
\sigma}(q^{n-1},\, q)$, for some $\sigma '\in U_{d}(k)$. Since two
Hopf algebras in the family $\mathrm{T}_{n,\,n}^{\,
\sigma}(q^{n-1},\,q)$ are isomorphic if and only if the
corresponding scalars $\sigma$ are equal, the proof is finished.
\end{proof}

Also as a consequence of \thref{clasificare} we can now compute
the automorphism groups of the Hopf algebras described in
\coref{descriere}.

\begin{theorem}\thlabel{auto}
Let $m$, $n \in \NN^{*} \setminus \{1\}$, such that $m \neq n$ and
$d = (m,\,n)$. Then, for any $q \in P_{m}(k) $, $\overline{q} \in
P_{n}(k) $ and $\sigma \in U_{d}(k)$ we have the following
isomorphisms of groups:
\begin{enumerate}
\item[1)] ${\rm Aut}_{{\rm Hopf}}\bigl(\mathrm{T}_{n,\,m}^{\, \sigma}(\overline{q},\,q)\bigl) \cong k^{*} \times k^{*} \cong {\rm Aut}_{{\rm Hopf}}\bigl(\mathrm{T}_{n,\,n}^{\, \sigma}(\overline{q},\,q)\bigl)$ where the second isomorphism holds if $q \neq \overline{q} $;
\item[2)] ${\rm Aut}_{{\rm Hopf}}\bigl(\mathrm{T}_{n,\,n}^{\, \sigma}(q,\, q)\bigl) \cong  \left\{
\begin{array}{rcl} k^{*} \times k^{*} & \mbox{if}& \sigma^{2} \neq 1\\ \bigl(k^{*} \times k^{*}\bigl) \rtimes_{\tau} \ZZ_{2} & \mbox{if}& \sigma^{2} = 1 \end{array}\right. $, where $\bigl(k^{*} \times k^{*}\bigl) \rtimes_{\tau} \ZZ_{2}$ is the semi-direct product associated to the action as automorphisms $\tau: \ZZ_{2} \to {\rm Aut} (k^{*} \times k^{*})$ given by $\tau(\overline{1})(a,\, b) = (b,\,a)$, for all $(a,\, b) \in k^{*} \times k^{*}$;
\item[3)] ${\rm Aut}_{{\rm Hopf}}\bigl(Q_{n}^{1}(q)\bigl) \cong \left\{
\begin{array}{rcl} k^{*}  & \mbox{if}& n \geqslant 3\\ k^{*} \rtimes_{\kappa} \ZZ_{2} & \mbox{if}& n = 2 \end{array}\right. $, where $k^{*} \rtimes_{\kappa} \ZZ_{2}$ is the semi-direct product associated to the action as automorphisms  $\kappa: \ZZ_{2} \to {\rm Aut} (k^{*})$ given by $\kappa(\overline{1})(a) = a^{-1}$, for all $a \in k^{*}$.
\end{enumerate}
\end{theorem}
\begin{proof}
The first two assertions are just easy consequences of
\thref{clasificare}. For 1), consider the bicrossed products
$\mathrm{T}_{n,\,m}^{\, \sigma}(\overline{q},\,q)$ and
respectively $\mathrm{T}_{n,\,n}^{\, \sigma}(\overline{q},\,q)$
with $q \neq \overline{q}$. Then, from the proof of the
aforementioned theorem we obtain that any automorphism of the
above Hopf algebras is of the form $\psi(a \bowtie g) = u(a)
\bowtie v(g)$ for all $a \in \mathbb{T}_{n^{2}}(\overline{q})$, $g
\in T_{m^{2}}(q)$, where the maps $u$ and $v$ are given as
follows:
\begin{eqnarray*}
&& u(X^{j} H^{i}) = \beta^{j}\, X^{j}\, H^{i},\,\, v(x^{l} h^{t})
= \eta^{l}\, x^{l}\, h^{t},\,\, {\rm with}\,\, \beta, \eta \in
k^{*}\\ && {\rm and}\,\, i,\,j \in \{0,\,1,\, \cdots,\, n-1\},\,\,
l,\,t \in \{0,\,1,\, \cdots,\, m-1\}.
\end{eqnarray*}
More precisely, the automorphism $\psi_{\beta,\, \eta}$
corresponding to a pair $(\beta,\, \eta) \in k^{*} \times k^{*}$
is given by:
$$
\psi_{\beta,\, \eta}(X) = \beta\, X,\,\, \psi_{\beta,\, \eta}(H) =
H,\,\, \psi_{\beta,\, \eta}(x) = \eta\, x,\,\, \psi_{\beta,\,
\eta}(h) = h.
$$
It is now straightforward to see that the automorphism group of
these Hopf algebras is just $k^{*} \times k^{*}$ with
componentwise multiplication.

2) Consider now the Hopf algebra $\mathrm{T}_{n,\,n}^{\,
\sigma}(q,\, q)$. Using again the proof of \thref{clasificare} we
obtain that the automorphisms of $\mathrm{T}_{n,\,n}^{\,
\sigma}(q,\, q)$ are of the form $\psi_{\beta,\, \eta}$ described
above with $(\beta,\, \eta) \in (k^{*} \times k^{*})$ and thus
${\rm Aut}_{{\rm Hopf}}\bigl(\mathrm{T}_{n,\,n}^{\, \sigma}(q,\,
q)\bigl) \cong k^{*} \times k^{*}$.

However, if $q^{2} = 1$, we have in addition the following family
of automorphisms denoted by $\varphi_{\zeta,\, \gamma}$, with
$(\zeta,\, \gamma) \in (k^{*} \times k^{*})$, and defined as
follows:
$$
\varphi_{\zeta,\, \gamma}(X) = \zeta\, x, \,\, \varphi_{\zeta,\,
\gamma}(H) = h,\,\, \varphi_{\zeta,\, \gamma}(x) = \gamma\, X,\,\,
\varphi_{\zeta,\, \gamma}(h) = H.
$$
Moreover, it can be easily seen that the
multiplication rules for the two families of automorphisms are as
follows:
\begin{eqnarray*}
&& \psi_{\beta,\, \eta} \circ \psi_{\beta ',\, \eta '} =
\psi_{\beta \beta',\, \eta \eta'},\,\,\, \varphi_{\zeta,\, \gamma}
\circ \varphi_{\zeta ',\, \gamma '} = \psi_{\zeta ' \gamma,\,
\zeta
\gamma '}\\
&& \varphi_{\zeta,\, \gamma} \circ \psi_{\beta,\, \eta} =
\varphi_{\beta \zeta,\, \eta \gamma},\,\,\, \psi_{\beta,\, \eta}
\circ \varphi_{\zeta,\, \gamma} = \varphi_{\zeta \eta,\, \gamma
\beta}.
\end{eqnarray*}
It is now straightforward to check that the map defined below is a
group isomorphism:
\begin{eqnarray*}
\Upsilon: {\rm Aut}_{{\rm Hopf}}\bigl(\mathrm{T}_{n,\,n}^{\,
\sigma}(q,\, q)\bigl) \to \bigl(k^{*} \times k^{*}\bigl)
\rtimes_{\tau} \ZZ_{2},\,\, \Upsilon(\psi_{\beta,\, \eta}) =
\bigl((\beta,\, \eta),\,\overline{0}\bigl),\,\,
\Upsilon(\varphi_{\zeta,\, \gamma}) = \bigl((\gamma,\, \zeta),\,
\overline{1}\bigl).
\end{eqnarray*}

3) Let $Q_{n}^{1}(q) =  \mathbb{T}_{n^{2}}(q^{n-1}) \bowtie T_{n^{2}}(q)$ be the bicrossed product corresponding to the matched pair in \thref{main}, 2), for $\alpha = 1$, namely:
\begin{eqnarray}
&& h \triangleleft H = h,\,\,\, h \triangleleft X = 0,\,\,\, x \triangleleft H = qx,\,\,\, x \triangleleft X = 1-h \eqlabel{mp11}\\
&& h \triangleright H = H, \,\,\, h \triangleright X = q X,\,\,\, x \triangleright H = 0,\,\,\, x \triangleright X = 1-H \eqlabel{mp22}.
\end{eqnarray}
As before, our approach relies on \thref{toatemorf}: any Hopf algebra map $\Gamma : Q_{n}^{1}(q)  \to Q_{n}^{1}(q)$ is parameterized by a quadruple $(u, p, r, v)$ consisting of unital coalgebra maps $u: \mathbb{T}_{n^{2}}(q^{n-1}) \to \mathbb{T}_{n^{2}}(q^{n-1}) $, $p: \mathbb{T}_{n^{2}}(q^{n-1}) \to T_{n^{2}}(q)$, $r: T_{n^{2}}(q) \to \mathbb{T}_{n^{2}}(q^{n-1}) $, $v: T_{n^{2}}(q) \to T_{n^{2}}(q)$ satisfying the compatibility conditions \equref{C1}-\equref{C8}. Exactly as in the proof of \thref{clasificare} we obtain, based on the compatibility condition \equref{C1} applied for $a = X$, the same nine possibilities for the pair of maps $(u,\,p)$. Moreover, again as in the proof of the aforementioned theorem, the compatibility conditions \equref{C3} and \equref{C4} give:
\begin{eqnarray}
u(X) H^{a} + p(X) \triangleright  H^{a} = q^{n-1} H^{a} \bigl(h^{b} \triangleright  u(X)\bigl)\eqlabel{c1}\\
\bigl(p(X)  \triangleleft  H^{a} \bigl) h^{b} = q^{n-1} \bigl(h^{b} \triangleleft  u(X) \bigl) h^{b} + q^{n-1} h^{b} p(X)\eqlabel{c2}
\end{eqnarray}
where this time $\triangleright$ and respectively $ \triangleleft $ are those given in \equref{mp11} and \equref{mp22}.

Suppose the maps $u$, $p$ are given as in I., i.e.:
\begin{eqnarray*}
u(H) = H,\,\, u(X) = \alpha (1 - H),\,\, p(H) = h,\,\, p(X) =  \alpha (1 - h),\,\, \alpha \in k.
\end{eqnarray*}
Then the compatibility condition \equref{c1} gives:
\begin{eqnarray*}
\alpha (1 - H) H + \alpha (1 - h) \triangleright H = q^{n-1} H \bigl(h \triangleright \alpha (1 - H)\bigl)
\end{eqnarray*}
which comes down to $\alpha = \alpha q^{n-1}$ and since $q^{n-1} \neq 1$ we get $\alpha = 0$. As we have seen before, this implies $\Gamma(X \bowtie 1) = 0$ so in this case $\Gamma$ is not an isomorphism. Using the same strategy as above, a rather long but straightforward computation based on the compatibility conditions \equref{c1} and \equref{c2} rules out the other cases as well except for the second one and the fourth one if $n = 2$.
Indeed, suppose the maps $u$, $p$ are given as in II., i.e.:
$$
u(H) = H,\,\, u(X) = \alpha (1 - H) + \beta X,\,\, p(H) = 1,\,\,
p(X) =  0,\,\, \alpha, \beta \in k.
$$
Then the compatibility condition \equref{c2} is trivially fulfilled while \equref{c1} gives:
$$
\bigl(\alpha (1-H)+ \beta X\bigl) H = q^{n-1} H\bigl(\alpha (1-H)+ \beta X\bigl)
$$
which comes down to $\alpha = \alpha q^{n-1}$ and since $q^{n-1}
\neq 1$ we get $\alpha = 0$. Thus $u(X) = \beta X$ and $p(H) = 1$.
Moreover, \equref{C3} and \equref{C4} yield the following general
formulas for $u$ and $p$:
\begin{eqnarray*}
&& u(X^{j} H^{i}) = \beta^{j}\, X^{j}\, H^{i},\,\, {\rm where}\,\, \beta \in k^{*},\,\, i,\,j \in \{0,\,1,\, \cdots,\, n-1\}\\
&& p(Y) = \epsilon(Y) 1,\,\, {\rm for} \,\, {\rm all} \,\, Y \in \mathbb{T}_{n^{2}}(q^{n-1}).
\end{eqnarray*}
Finally, suppose the maps $u$, $p$ are given as in IV., i.e.:
$$
u(H) = 1,\,\, u(X) = 0,\,\, p(H) = h,\,\, p(X) =  \alpha (1 - h) + \beta x,\,\, \alpha, \beta \in k.
$$
Then, \equref{c1} is trivially fulfilled while \equref{c2} gives:
$$
\bigl(\alpha (1-h)+ \beta x\bigl) h = q^{n-1} h\bigl(\alpha (1-h)+ \beta x\bigl)
$$
which comes down to $\alpha = \alpha q^{n-1}$ and $q \beta =
q^{n-1} \beta$.  Therefore, we must have $\alpha = 0$ and if $n
\geqslant 3$ we also have $\beta = 0$ which we know it implies
that $\Gamma$ is not an isomorphism. However, for $n = 2$ the
second equality is trivially fulfilled and we obtain $p(X) = \beta
x$ for some $\beta \in k^{*}$. Then, using again \equref{C3} and
\equref{C4} we get:
\begin{eqnarray*}
&&p(X^{i}H^{j}) = \beta^{i} x^{i}h^{j},\,\,{\rm where}\,\, \beta \in k^{*},\,\, i,\,j \in \{0,\,1,\, \cdots,\, n-1\},\,\, \\
&&u(Y) = \epsilon(Y)1,\,\, {\rm for\,\, all}\,\,Y \in \mathbb{T}_{n^{2}}(q^{n-1}).
\end{eqnarray*}

We employ the same strategy in order to determine the pair of maps
$(r, \,v)$ but relying this time on the compatibility conditions
\equref{C2}, \equref{C5} and respectively \equref{C6}. Then for
any $n \geqslant 3$ we obtain the following pair of maps:
\begin{eqnarray*}
&& v(x^{j} h^{i}) = \eta^{j}\, x^{j}\, h^{i},\,\, {\rm where}\,\, \eta \in k^{*},\,\, i,\,j \in \{0,\,1,\, \cdots,\, n-1\}\\
&& r(y) = \epsilon(y) 1,\,\, {\rm for} \,\, {\rm all} \,\, y \in T_{n^{2}}(q)
\end{eqnarray*}
and in the case when $n=2$ we have, in addition, the following:
\begin{eqnarray*}
&& v(y) = \epsilon(y) 1,\,\, {\rm for} \,\, {\rm all} \,\, y \in T_{n^{2}}(q)\\
&& r(x^{j} h^{i}) = \eta^{j}\, X^{j}\, H^{i},\,\, {\rm where}\,\, \eta \in k^{*},\,\, i,\,j \in \{0,\,1,\, \cdots,\, n-1\}.
\end{eqnarray*}
Since, as mentioned before, $Q_{2}^{1}(-1)$ coincides with the
Hopf algebra $\mathcal{H}_{16, 1}$ whose automorphism group was
described in detail in \cite{Gabi} we will not provide the
complete computations for the case $n=2$. We only included the
details above in order to emphasize the difference between the two
cases, namely $n=2$ and respectively $n \geqslant 3$, and to
explain why the two automorphism groups are different. For the
remaining of the proof we will focus on the case $n \geqslant 3$.
Considering the above discussion, we are left to check the
compatibility conditions \equref{C7} and \equref{C8} for the
quadruple $(u, \,p,\, r,\, v)$ given as follows:
\begin{eqnarray*}
&& u(X^{j} H^{i}) = \beta^{j}\, X^{j}\, H^{i},\,\, v(x^{j} h^{i}) = \eta^{j}\, x^{j}\, h^{i},{\rm where}\,\, \beta, \eta \in k^{*},\,\, i,\,j \in \{0,\,1,\, \cdots,\, n-1\}\\
&& p(Y) = \epsilon(Y) 1,\,\, r(y) = \epsilon(y) 1,\,\, {\rm for} \,\, {\rm all} \,\, Y \in \mathbb{T}_{n^{2}}(q^{n-1}),\,\, y \in T_{n^{2}}(q).\\
\end{eqnarray*}
As $r$ and $p$ are trivial maps, the two aforementioned compatibilities reduce to the following:
\begin{eqnarray*}
v(g) \triangleright u(b) = u(g \triangleright b),\,\,\, v(g)
\triangleleft u(b) = v(g \triangleleft b).
\end{eqnarray*}
Now it is straightforward to see that both compatibilities are trivially fulfilled for $(g,\,b) \in \{(h,\, H),\, (x,\, H),\, (h,\,X)\}$. Finally, if $(g,\,b) = (x,\, X)$ we obtain $\eta \beta (1-H) = 1-H$ and respectively $\eta \beta (1-h) = 1-h$ and thus $\eta = \beta^{-1}$. Moreover, the automorphism induced by the quadruple $(u, \,p,\, r,\, v)$ depicted above, which we will denote by $\varphi_{\beta}$ for any $\beta \in k^{*}$, is given as follows:
$$\varphi_{\beta}(H) = H,\,\,\, \varphi_{\beta}(X) = \beta\,X,\,\,\, \varphi_{\beta}(h) = h,\,\,\, \varphi_{\beta}(x) = \beta^{-1}\,x.$$
We can now conclude that ${\rm Aut}_{{\rm Hopf}}\bigl(Q_{n}^{1}(q)\bigl) \cong k^{*}$ with the obvious isomorphism which sends any automorphism $\varphi_{\beta}$ to $\beta \in k^{*}$. Together with \cite[Corollary 2.5]{Gabi} we obtain the desired conclusion.
\end{proof}

\end{document}